\documentclass[11pt]{amsart}
\usepackage{graphicx}
\usepackage[usenames]{color}
\usepackage{amssymb,amscd}
\newcommand\Z{\mathbb Z}
\newcommand\Q{\mathbb Q}
\newcommand\R{\mathbb R}

\newcommand\Rinf{R_{\infty}}

\newcommand\vp{\varphi}
\newcommand\Op{\overline{\varphi}}
\newcommand\ph{\varphi}
\newcommand\Ai{A_i \bigoplus A_{j-i}}
\newtheorem{theorem}{Theorem}[section]
\newtheorem{proposition}[theorem]{Proposition}
\newtheorem{lemma}[theorem]{Lemma}
\newtheorem{corollary}[theorem]{Corollary}

\newtheorem{example}[theorem]{Example}

\title{The geometry of twisted conjugacy classes in wreath products}
\author{Jennifer Taback and Peter Wong}
\dedicatory{To Robert Zimmer on the occasion of his 60th birthday.}
\address{Department of Mathematics,
Bowdoin College, Brunswick, ME 04011} \email{jtaback@bowdoin.edu}
\thanks{The first author acknowledges support from
NSF grant DMS-0604645, and would like to thank David Fisher and
Kevin Whyte for many useful conversations about this paper, and
Kevin Wortman for comments.}
\address{Department of Mathematics, Bates College, Lewiston, ME 04240} \email{pwong@bates.edu}
\date{\today}
\begin{document}

\maketitle

\begin{abstract}
We give a geometric proof based on recent work of Eskin, Fisher and
Whyte that the lamplighter group $L_n$ has infinitely many twisted
conjugacy classes for any automorphism $\vp$ only when $n$ is
divisible by 2 or 3, originally proved by Gon\c calves and Wong.  We
determine when the wreath product $G \wr \Z$ has this same property
for several classes of finite groups $G$, including symmetric groups
and some nilpotent groups.
\end{abstract}

\section{Introduction}
We say that a group $G$ has {\em property $\Rinf$} if any
automorphism  $\vp$ of $G$ has an infinite number of $\vp$-twisted
conjugacy classes.  Two elements $g_1,g_2 \in G$ are $\vp$-twisted
conjugate if there is an $h \in G$ so that $hg_1\vp(h)^{-1} = g_2$.
The study of this property is motivated by topological fixed point
theory, and is discussed below.  Groups with property $\Rinf$
include
\begin{enumerate}
\item Baumslag-Solitar groups $BS(m,n) = \langle a,b | ba^mb^{-1} = a^n \rangle$ except for $BS(1,1)$. \cite{FGo}
\item Generalized Baumslag-Solitar (GBS) groups, that is, finitely generated groups which act on a tree with all edge and vertex stabilizers infinite cyclic as well as any group quasi-isometric to a GBS group. \cite{L,TW2}
\item Lamplighter groups $\Z_n \wr \Z$ iff $2|n$ or $3|n$. \cite{GW1}
\item Some groups of the form $\Z^2 \rtimes \Z$. \cite{GW2}
\item Non-elementary Gromov hyperbolic groups. \cite{LL,F}
\item The solvable generalization $\Gamma$ of $BS(1,n)$ which is defined by the short exact sequence
$1 \rightarrow \Z[\frac{1}{n}] \rightarrow \Gamma \rightarrow \Z^k
\rightarrow 1$
 as well as any group quasi-isometric to $\Gamma$. \cite{TW}
\item The mapping class group. \cite{FG}
\item Relatively hyperbolic groups. \cite{FG}
\item Any group with a {\em characteristic} and non-elementary action on a Gromov hyperbolic space.
\cite{TWh,FG}
\end{enumerate}
These results fall into two categories: those which show the
property holds using a group presentation, and those which show that
property $\Rinf$ is geometric in some way, whether invariant under
quasi-isometry, or dependent on an action of a group on a particular
space.  The final statement on the list is the most general, relying
only on the existence of a certain group action. This was proven
independently in \cite{TWh} and in the Appendix of \cite{FG} using similar methods, and provides a wealth
of examples of groups with this property.

The study of the finiteness of the number of twisted conjugacy
classes arises in Nielsen fixed point theory. Given a selfmap
$f:X\to X$ of a compact connected manifold $X$, the nonvanishing of
the classical Lefschetz number $L(f)$ guarantees the existence of
fixed points of $f$. Unfortunately, $L(f)$ yields no information
about the size of the set of fixed points of $f$. However, the
Nielsen number $N(f)$, a more subtle homotopy invariant, provides a
lower bound on the size of this set. For $\dim X\ge 3$, a classical
theorem of Wecken  \cite{Wecken2} asserts that $N(f)$ is a sharp
lower bound on the size of this set, that is, $N(f)$ is the minimal
number of fixed points among all maps homotopic to $f$. Thus the
computation of $N(f)$ is a central issue in fixed point theory.

Given an endomorphism $\ph :\pi \to \pi$ of a group $\pi$, the
$\ph$-twisted conjugacy classes are the orbits of the (left) action
of $\pi$ on $\pi$ via $\sigma \cdot \alpha \mapsto \sigma \alpha
\ph(\sigma)^{-1}$.  Given a selfmap $f:X\to X$ of a compact
connected polyhedron, the fixed point set $Fixf=\{x\in X|f(x)=x\}$
is partitioned into fixed point classes, which are identical to the
$\ph$-twisted conjugacy classes where $\ph=f_{\sharp}$ is the
homomorphism induced by $f$ on the fundamental group $\pi_1(X)$. The
Reidemeister number $R(f)$ is the number of $\ph$-twisted conjugacy
classes, and is an upper bound for $N(f)$.  When $R(f)$ is finite,
this provides additional information about the cardinality of the
set of fixed points of $f$.

For a certain class of spaces, called {\em Jiang spaces}, the
vanishing of the Lefshetz number implies that $N(f) = 0$ as well,
and a nonzero Lefshetz number, combined with a finite Reidemeister
number, implies that $N(f)=R(f)$. As the Reidemeister number is much
easier to calculate than the Lefshetz number, this provides a
valuable tool for computing the cardinality of the set of fixed
points of the map $f$. Jiang's results have been extended to
selfmaps of simply connected spaces, generalized lens spaces,
topological groups, orientable coset spaces of compact connected Lie
groups, nilmanifolds, certain $\mathcal C$-nilpotent spaces where
$\mathcal C$ denotes the class of finite groups, certain
solvmanifolds and infra-homogeneous spaces (see, e.g., \cite{Wo2},
\cite{Wo3}). Groups $G$ which satisfy property $\Rinf$, that is,
every automorphism $\vp$ has $R(\vp) = \infty$, will never be the
fundamental group of a manifold which satisfies the conditions
above.

In this paper, we study the (in)finiteness of the number of
$\ph$-twisted conjugacy classes for automorphisms $\ph:G \wr \Z \to
G \wr \Z$ where $G\wr \mathbb Z$ is the restricted wreath product of
a finite group $G$ with $\mathbb Z$.  When $G = \Z_n$, so that $G
\wr \Z$ is a ``lamplighter" group, we prove that $G \wr \Z$ has
property $\Rinf$ iff $2$ or $3$ divides $n$ using the geometry of
the Cayley graph of these groups.  Our proofs rely on recent work of
A. Eskin, D. Fisher and K. Whyte identifying all quasi-isometries of
$L_n$, and classifying these groups up to quasi-isometry. \cite{EFW}
Our theorem was originally proven in \cite{GW1} using algebraic
methods.  The geometric techniques used below extend to certain
other groups of this form, but combining geometric and algebraic
methods we determine several larger classes of finite groups $G$ for
which $G \wr \Z$ has property $\Rinf$.  In particular, $S_n \wr \Z$
has this property, for $n \geq 5$, yielding the corollary that every
group of the form $G \wr \Z$ can be embedded into a group which has
property $\Rinf$.

We note that property $\Rinf$ is not geometric in the sense that it
is preserved under quasi-isometry.  Namely, $\mathbb Z_4 \wr \mathbb
Z$ is quasi-isometric to $(\mathbb Z_2)^2 \wr \mathbb Z$ by
\cite{EFW} but according to \cite{GW1} the former has property
$R_{\infty}$ while the latter does not.  It also fails to be
geometric when one considers cocompact lattices in $Sol$. Let $A, \
B \in GL(2,\Z)$ be matrices whose traces have absolute value greater
than two. We know that $\Z^2 \rtimes_A \Z$ and $\Z^2 \rtimes_B \Z$
are always quasi-isometric, as they are both cocompact lattices in
$Sol$, but they may not both have property $\Rinf$.  However, there
are classes of groups for which this property is invariant under
quasi-isometry; these include the solvable Baumslag-Solitar groups,
their solvable generalization $\Gamma$ given above and the
generalized Baumslag-Solitar groups.

\eject

\section{Background on twisted conjugacy}
\subsection{Twisted conjugacy}\label{sec:twistedconj}

We say that a group $G$ has {\em property $\Rinf$} if, for any  $\vp
\in Aut(G)$, we have $R( \vp) = \infty$.  The main technique we use
for computing $R(\vp)$ is as follows.  We consider groups which can
be expressed as group extensions, for example $1 \rightarrow A
\rightarrow B \rightarrow C \rightarrow 1$.  Suppose that an
automorphism $\vp \in Aut(B)$ induces the following commutative
diagram, where the vertical arrows are group homomorphisms, that is,
$\ph|_A = \ph'$ and ${\overline \ph}$ is the quotient map induced by
$\ph$ on $C$:

\begin{equation}\label{general-short-exact}
\begin{CD}
    1   @>>> A    @>{i}>>  B @>{p}>>      C @>>> 1 \\
    @.  @V{\ph'}VV      @V{\ph}VV   @V{\overline \ph}VV @.\\
    1   @>>> A    @>{i}>>  B @>{p}>>      C @>>> 1
 \end{CD}
\end{equation}

Then we obtain a short exact sequence of sets and corresponding
functions ${\hat i}$ and ${\hat p}$:
\begin{equation}\label{eqn:seq-of-sets}
\mathcal R(\ph') \stackrel{\hat i}{\to} \mathcal R(\ph)
\stackrel{\hat p}{\to} \mathcal R(\overline \ph)
\end{equation}
where if $\bar 1$ is the identity element in $C$, we have
$\hat{i}(\mathcal R(\ph'))={\hat p}^{-1}([\bar 1])$,  and $\hat p$
is onto. To ensure that both $\varphi'$ and $\overline \ph$ are both
automorphisms, we need the following lemma.  Recall that a group is
{\em Hopfian} if every epimorphism is an automorphism.

\begin{lemma}\label{lemma:auto}
If $C$ is Hopfian, then $\ph'\in Aut(A)$ and $\overline \ph \in Aut(C)$.
\end{lemma}
\begin{proof}
Since $\ph$ is an automorphism, the commutativity of
\eqref{general-short-exact} implies that $\ph'$ is injective and
$\overline \ph$ is surjective. Since $C$ is Hopfian, $\overline \ph
\in Aut(C)$. Suppose $a\in A-\ph'(A)$. It follows that
$\ph^{-1}(i(a))$ is an element of $B-i(A)$ so $p(\ph^{-1}(i(a)))$ is
non-trivial in $C$. This means that $$\overline \ph
(p(\ph^{-1}(i(a))))=p\ph(\ph^{-1}(i(a)))=p(i(a))$$ is non-trivial, a
contradiction. Thus, such an element $a$ cannot exist and so $\ph'$
is also surjective.
\end{proof}

The following result is straightforward and follows from
more general results discussed in \cite{Wo}.

\begin{lemma}\label{lemma:reid}
Given the commutative diagram labeled \eqref{general-short-exact} above,
\begin{enumerate}
\item if $R(\overline \ph)=\infty$ then $R(\ph)=\infty$,

\item if $C$ is finite or $Fix(\Op) = 1$, and $R(\ph')=\infty$ then $R(\ph)=\infty$.
\end{enumerate}
When $C \cong \Z$, we have one of two situations:
\begin{enumerate}
\item[(3)] the map $\Op$ is the Identity, in which case $R(\Op) = \infty$ and hence $R(\vp) = \infty$, or
\item[(4)] $\Op(t) = t^{-1}$ and $R(\Op) = 2$. In this case, $R(\vp) <
\infty$ iff $R(\vp') < \infty$ and $R(t \cdot \vp') < \infty$, in
which case $R(\vp) = R(\vp') + R(t \cdot \vp')$.
\end{enumerate}
\end{lemma}

To show that a group $G$ does not have property $\Rinf$, it suffices
to produce a single example of $\vp \in Aut(G)$ with $R(\vp) <
\infty$.  When $G$ additionally has a semidirect product structure,
such an automorphism is often constructed as follows.  Write $G = A
\rtimes B$ as $1 \rightarrow A \rightarrow G \rightarrow B
\rightarrow 1$, with $\Theta: B \rightarrow Aut(A)$.  We can use
automorphisms $\vp': A \rightarrow A$ and $\Op: B \rightarrow B$ to
construct the following commutative diagram defining $\vp \in
Aut(G)$, provided that

\begin{equation}\label{eqn:compatibility}
\vp'(\Theta(b)(a)) = \Theta(\Op(b))(\vp'(a))
\end{equation}
 for $b \in B$ and $a
\in A$:
\begin{equation}\label{short-exact}
\begin{CD}
    1   @>>> A    @>{i}>>  G @>{p}>>      B @>>> 1 \\
    @.  @V{\ph'}VV      @V{\ph}VV   @V{\overline \ph}VV @.\\
    1   @>>> A    @>{i}>>  G @>{p}>>      B @>>> 1
 \end{CD}
\end{equation}
Through careful selection of the maps $\vp'$ and $\Op$, we can
sometimes create an example of an automorphism of $G$ with finite
Reidemeister number using the above diagram and Lemma
\ref{lemma:reid}. We should point however that the map $\ph$ so constructed using $\ph'$ and $\overline \ph$ is not unique.

\section{Lamplighter groups and Diestel-Leader graphs}
We begin with a discussion of the geometry of the lamplighter groups
$L_n = \Z_n \wr \Z$ and later address more general groups $G \wr \Z$
where $G$ is any finite group.

The standard presentation of $L_n$ is $\langle a,t | a^n = 1, \
[a^{t^i},a^{t^j}]=1 \rangle$ where $x^y$ denotes the conjugate
$yxy^{-1}$.  Equivalently, we recall that a wreath product is simply
a certain semidirect product, and write $L_n$ as $\left(
\bigoplus_{i=- \infty}^{\infty} A_i\right) \rtimes \Z$, with each
$A_i \cong Z_n$, fitting into the split short exact sequence
$$0 \rightarrow  \bigoplus_{i=- \infty}^{\infty} A_i \rightarrow L_n \rightarrow \Z
\rightarrow 0$$ where the generator of $\Z$ is taken to be  $t$ from
the presentation given above.

The ``lamplighter" picture of elements of this group is the
following.  Take a bi-infinite string of light bulbs placed at
integer points along the real line, each of which has $n$ states
corresponding to the powers of the generator $a$ of $\Z_n$, and a
cursor, or lamplighter, indicating the particular bulb under
consideration.  The action of the group on this picture is such that
$t$ moves the cursor one position to the right, and $a$ changes the
state of the current bulb under consideration.  Thus each element of
$L_n$ can be interpreted as a series of instructions to illuminate a
finite collection of light bulbs to some allowable states, leaving
the lamplighter at a fixed integer.

Understanding group elements via this picture, the generator $a_j =
t^j a t^{-j}$ of $A_j$ in $\bigoplus_{i=- \infty}^{\infty} A_i$ has
a single bulb in position $j$ illuminated to state $a$, and the
cursor at the origin of $\Z$. This leads to two possible normal
forms for $g \in L_n$, as described in \cite{CT}, separating the
word into segments corresponding to bulbs indexed by negative and
non-negative integers:
$$rf(g)=a^{e_{1}}_{i_1} a^{e_{2}}_{i_2} \ldots a^{e_{k}}_{i_k} a^{f_{1}}_{-j_1} a^{f_{2}}_{-j_2} \ldots a^{f_{l}}_{-j_l} t^{m}$$
or
$$lf(g)=a^{f_{1}}_{-j_1} a^{f_{2}}_{-j_2} \ldots a^{f_{l}}_{-j_l} a^{e_{i}}_{i_1} a^{e_{i}}_{i_2} \ldots a^{e_{k}}_{i_k}   t^{m}
$$
with $ i_k > \ldots i_2 > i_1 \geq 0 $ and $ j_l > \ldots j_2 > j_1
> 0 $ and $e_i, \ f_j$ in  the range $\{ -h,\cdots, h \}$,
where $h$ is the integer part of $\frac{n}{2}$.  When $n$ is even,
we omit $a^{-h}$ to ensure uniqueness, since $a^h=a^{-h}$ in
$\Z_{2h}$.  If we consider the normal form as a series of
instructions for acting on the lamplighter picture to create these
group elements, then the normal form $rf(g)$ first illuminates bulbs
at or to the right of the origin, and $lf(g)$ first illuminates the
bulbs to the left of the origin.  It is proven in \cite{CT} that the
word length of $g \in L_n$ with respect to the finite generating set
$\{a,t\}$, using the notation of either normal form given above, is
$$\sum e_{i}+\sum f_{j} + min\{2 j_l+i_k + | m-i_k|, 2
i_k+j_l+|m+j_l|\}.$$

We note that the lamplighter picture of $L_n$ does not yield the
Cayley graph of the group; to obtain a useful and understandable
Cayley graph for this group we introduce Diestel-Leader graphs in \S
\ref{sec:DL} below. However, to use a Diestel-Leader graph as the
Cayley graph of $L_n$, we must consider the generating set $\{ t,ta,
ta^2, \cdots ,ta^{n-1} \}$ for $L_n$.

There is another generating set for $L_2$ worth noting. R.
Grigorchuk and A. Zuk in \cite{GZ} show that this group is an
example of an automata group.  The natural generating set arising
from this interpretation is $\{a,ta\}$.  Grigorchuk and Zuk compute
the spectral radius with respect to this generating set and find
that it is a discrete measure.

\subsection{Diestel-Leader graphs}\label{sec:DL}

We now describe explicitly the Cayley graph for the lamplighter
group $L_m$ with respect to one particular generating set.  This
Cayley graph is an example of a Diestel-Leader graph, which we
define in full generality as follows.

For positive integers $m \leq n$, the Diestel-Leader graph $DL(m,n)$
is a subset of the product of the regular trees of valence $m+1$ and
$n+1$, which we denote respectively as $T_1$ and $T_2$.  We orient
these trees so that each vertex has $m$ (resp. $n$) outgoing edges.
By fixing a basepoint, we can use this orientation to define a
height function $h_i:T_i \rightarrow \Z$.  The Diestel-Leader graph
$DL(m,n)$ is defined to be the subset of $T_1 \times T_2$ for which
$h_1+h_2 = 0$.  This definition lends itself to the following
pictorial representation. Each tree has a distinguished point at
infinity which determines the height function.  If we place these
points for the two trees at opposite ends of a page, the
Diestel-Leader graph can be seen as those pairs of points, one from
each tree, lying on the same horizontal line.  See Figure
\ref{fig:DL33} for a portion of the graph $DL(3,3)$.

\begin{figure}[ht!]
  \includegraphics[scale=.25]{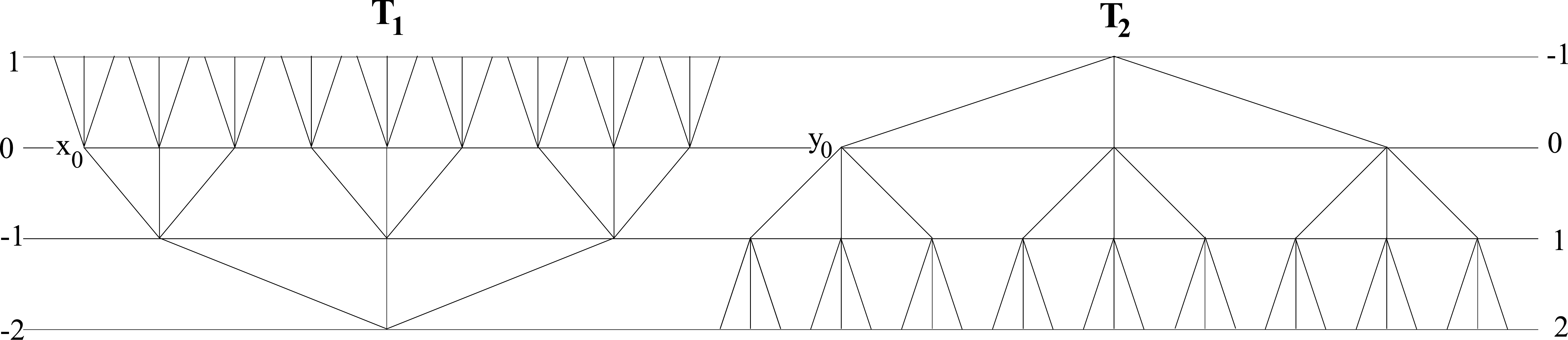}
  \caption{Part of the Cayley graph $DL(3,3)$ of $L_3$ with the point $Id = (x_0,y_0)$ labeled.
  The integers on the edges of the diagram represent the height in each tree.  The point $(x_0,y_0)$
  represents the identity element of the group.}
  \label{fig:DL33}
\end{figure}

When $n=m$, we will see that $DL(m,m)$ is the Cayley graph of the
lamplighter group $L_m$ with respect to a particular generating set.
This fact was first noted by Moeller and Neumann \cite{MN}.  Both
Woess \cite{Woe} and Wortman \cite{Wor} describe slightly different
methods of understanding this model of the lamplighter groups; our
explanation is concrete in a different way than either of theirs.

The group $L_m$ is often presented by $L_m \cong \langle a,t | a^n =
1, \ [a^{t^i},a^{t^j}]=1 \rangle$.  However, the Cayley graph
resulting from this presentation is rather untractable.  When we
take the generating set $\{t,ta,ta^2, \cdots ,ta^{m-1} \}$ we obtain
the Cayley graph $DL(m,m)$ for $L_m$.  Thus a group element $g \in
L_m$ corresponds to a pair of points $(c,B)$ where $c \in T_1$ and
$B \in T_2$ so that $h_1(c) + h_2(B) = 0$.  Our notational
convention will be to use lower case letters for vertices in $T_1$
and capital letters for vertices in $T_2$, with the exception of the
identity, which we always denote $(x_0,y_0)$.

The action of the group $L_m$ on $DL(m,m)$ is as follows:
\begin{enumerate}
\item $t$ translates up in height in $T_1$ and down in height in
$T_2$, so that the condition $h_1 + h_2 = 0$ is preserved, and
\item if $(v,W)$ is a point in $DL(m,m)$, let $v_0,v_1, \cdots ,v_{m-1}$ be the
vertices in $T_1$ adjacent to $v$ with $h_1(v_i) = h_1(v) + 1$. Then
$a$ performs a cyclic rotation among these vertices, so that $a
\cdot (v_i,W)  = (v_{(i+1)(\text{mod} \ m)},W)$. Conjugates of $a$
by powers of $t$ perform analogous rotations on the vertices with a
common parent in $T_2$ while fixing the coordinate in $T_1$.
\end{enumerate}

We now describe how to identify group elements with vertices of
$DL(m,m)$. Express $g \in L_m$ using the generating set
$\{t,ta,ta^2, \cdots ,ta^{m-1} \}$ and inverses of these elements.
For example, let $m=3$ and take $a=t^{-1}(ta)$.  We describe a path
in $DL(m,m)$ from the identity to $a$, using the expression
$a=t^{-1}(ta)$ and the action of the group on the graph.  We
introduce coordinates on the trees as needed, but always designate
the identity as $(x_0,y_0)$. From a given vertex $(x,Y) \in
DL(m,m)$, there are edges emanating from this vertex labeled by all
these generators and their inverses. The edge labeled by a generator
of the form $ta^i$ for $i=0,1, \cdots ,m-1$ leads to a point
$(x',Y')$ with $h_1(x') = h_1(x) + 1$ and $h_2(Y') = h_2(Y)-1$, and
the edge labeled by a generator of the form $(ta^i)^{-1}$ leads to a
point $(x',Y')$ with $h_1(x') = h_1(x) - 1$ and $h_2(Y') =
h_2(Y)+1$.

Notice that when leaving $(x,Y)$ via an edge labeled $(ta^i)^{-1}$,
there is only one possible vertex in $T_1$ adjacent to the vertex
$x$ at lower height.  When tracing out a path from the identity to a
particular element, we must remember which path was taken when it
seems like several generators correspond to traversing the same edge
in one of the two trees.  This will be illustrated in an example
below.

We begin by finding the group element $a = t^{-1}(ta)$ as a vertex
in $DL(3,3)$. We read the generators from left to right as a series
of instructions for forming a path in $DL(3,3)$ from the identity to
$a$.  Using the group action on $DL(3,3)$ described above, the
generator $t$ always increases the height function in $T_1$ while
decreasing it in $T_2$, and the action of $a$ (resp. $a^{-1}$) is to
cyclically rotate the edge we traverse in $T_1$ (resp. $T_2$).

For example, in $DL(3,3)$ we trace a path to $a = t^{-1}(ta)$ in
each tree as follows.  The edge emanating from $Id = (x_0,y_0)$
labeled $t^{-1}$ must decrease the height in $T_1$ and increase it
in $T_2$;  This determines a vertex $x'$ in $T_1$ and $Y'$ in $T_2$.
The edge emanating from $(x',Y')$ labeled by $ta$ must increase the
height function in $T_1$ and decrease it in $T_2$.  There is a
unique vertex in $T_2$ satisfying the conditions necessary for a
coordinate in the terminus of this edge, namely $y_0$.  In $T_1$,
since the initial edge was labeled $t^{-1}$, we do not traverse up
that edge, but rather the cyclically adjacent edge, ending at a
point whose coordinates we label $(x_1,y_0)$.  Using the coordinates
in Figure \ref{fig:DL33a} below, we see that $a^2 = t^{-1}(ta^2) =
(x_2,y_0)$.

\begin{figure}[h!]
  \includegraphics[scale=.3]{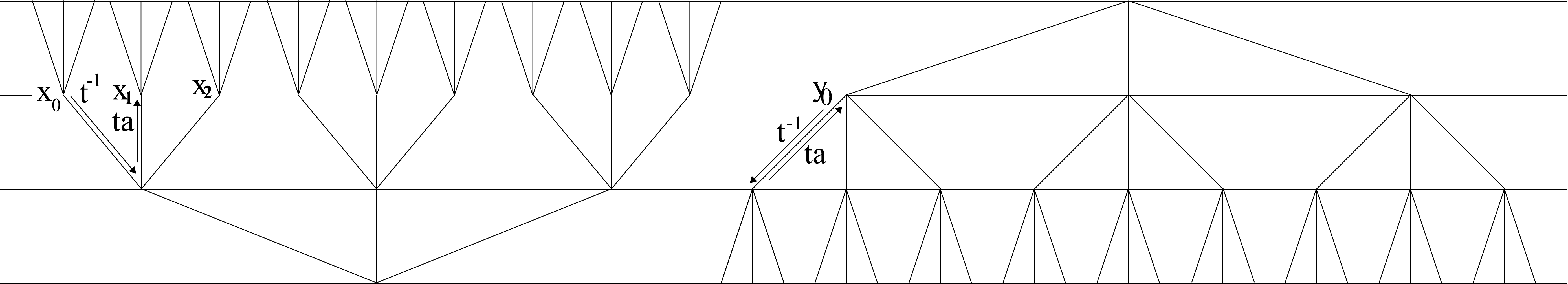}
  \caption{The path in the Cayley graph $DL(3,3)$ of $L_3$ from $Id = (x_0,y_0)$
  to the group element $a = t^{-1}(ta)$.}
  \label{fig:DL33a}
\end{figure}

For a more involved example, consider the group elements $t^{-1}
(ta^2)^{-1} t(ta)$ and $(ta)^{-1} (ta^2)^{-1} t(ta)$.  In Figure
\ref{fig:2paths}, we trace out the paths in $T_1$ leading from $x_0$
to the first coordinate of these two elements in $DL(3,3)$.  This
example illustrates the cyclic labeling of edges determined by the
downward path of edges and the action of $a$ on the graph. It is not
hard to see that the $T_2$ coordinate of both of these group
elements is $y_0$.

\begin{figure}[h!]
  \includegraphics[scale=.55]{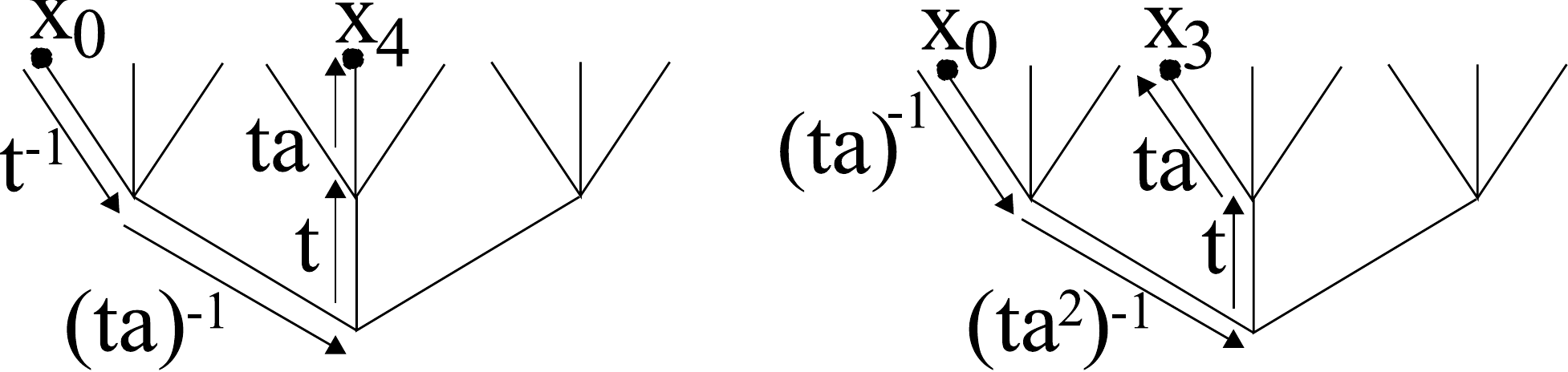}
  \caption{The $T_1$ coordinates of the paths in $DL(3,3)$ leading from the identity $(x_0,y_0)$
  to the group elements $t^{-1}(ta^2)^{-1} t(ta)$ and $(ta)^{-1} (ta^2)^{-1} t(ta)$, respectively.  The
  $T_2$ coordinates of the final point in each path is $y_0$.}
  \label{fig:2paths}
\end{figure}

Writing $L_m$ as a group extension: \begin{equation}\label{eqn:Ln}0
\rightarrow \bigoplus_{i=- \infty}^{\infty} A_i \rightarrow L_n
\rightarrow \Z \rightarrow 0\end{equation} where $A_i \cong \Z_m$,
we see that the map onto $\Z$ is determined by the exponent sum on
all instances of the generator $t$ in any word representing $g \in
L_m$, using the generating set $\{t,a\}$. Thus the kernel of this
map is simply those group elements in which there are equal numbers
of $t$ and $t^{-1}$.  It is easy to see that when we change to the
generating set $\{t,ta,\cdots ,ta^{n-1}\}$ these elements can be
characterized in the same manner.  Thus in $DL(m,m)$ the points
corresponding to the group elements in the kernel of this map again
have coordinates with height $0$ in both trees.

To understand this picture more thoroughly, we identify in $DL(m,m)$
the points corresponding to the different factors of $A_i \cong
\Z_m$. First suppose that $i \leq 0$.  The points in $A_i$ all have
the form $t^{-i} a^k t^i$ with respect to the $\{t,a\}$ generating
set for $L_m$, for $0 \leq k \leq m-1$, and $t^{-(i+1)} (ta^k) t^i$
with respect to the $\{t,ta,ta^2, \cdots ,ta^{m-1} \}$ generating
set. Thus they are easy to find in $DL(m,m)$: the coordinate in
$T_2$ will be $y_0$, and in $T_1$, follow the unique path which
takes $i+1$ edges decreasing in height, cyclically rotate over $k$
edges, proceed up in height one edge, then regardless of $k$,
proceed up the identical path in each subtree labeled by $t$ at each
step until you reach height zero.  The factors of $A_i$ for $i > 0$
are reversed in that points in such a factor all have the first
coordinate in $T_1$ equal to $x_0$ and the second coordinate is
found in $T_2$ in an analogous manner.

\begin{center}
\begin{figure}[h!]
  \includegraphics[scale=.35]{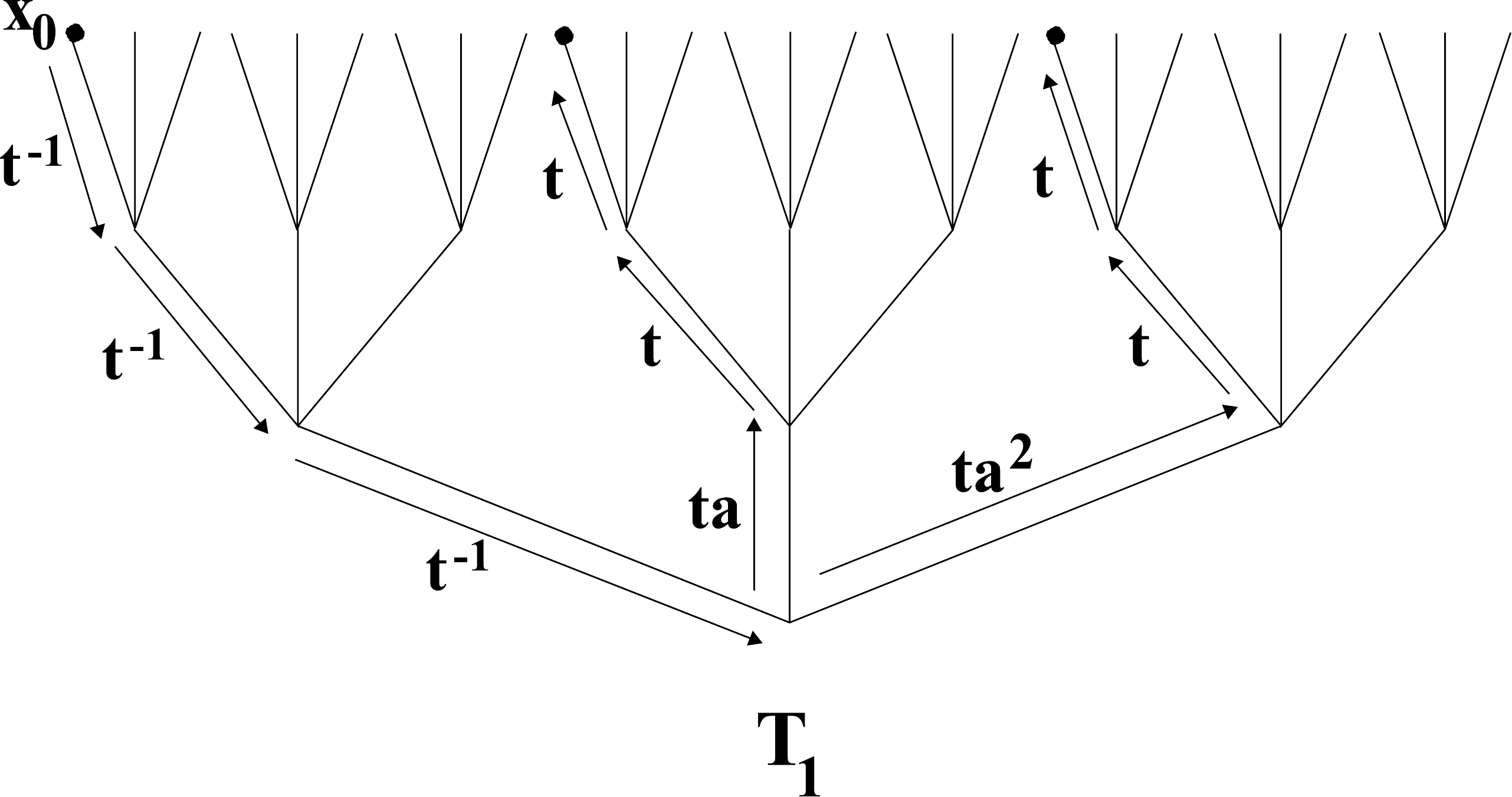}
  \caption{Paths from the identity to the $T_1$ coordinates of the points
  of $A_3$ in $\bigoplus_{i=-\infty}^{\infty} A_i$ where $A_i \cong \Z_3$ for all $i$.}
  \label{fig:block}
\end{figure}
\end{center}

Using the lamplighter picture to understand the group, a natural
subset to identify is the collection of elements with bulbs
illuminated (to any state) within a given set of positions, say
between $-i$ and zero, for some $i \geq 0$.  The situation we
describe below is completely analogous for $i <0$, with the two
trees interchanged. Let $l$ be the line in $DL(m,m)$ which is the
orbit of the identity under the group generator $t$.  Viewing $T_1$
and $T_2$ simply as trees, and not as part of $DL(m,m)$, the line
$l$ in $DL(m,m)$ determines a line $l_i \subset T_i$ for each $i$,
for $i=1,2$.

When $-i \leq 0$, to determine this set of points, we recall that in
the generating set corresponding to $DL(m,m)$, any instance of a
generator $(ta^k)^{\pm 1}$ for $k$ strictly greater than zero will
illuminate a bulb in the lamplighter picture for the group.  The
position of this bulb is determined by the exponent sum on all
instances of the generator $t$, whether as part of $(ta^k)$ or not,
within the word up until the string $a^k$ is read.  Thus all group
elements of this form can be expressed as follows:
\begin{equation}\label{eqn:lightson} t^{-(i+1)} \Pi_{j=1}^{j+1} \alpha_j\end{equation}
 where $\alpha_i $ or
$\alpha_i^{-1}$ lies in $\{t,ta,ta^2, \cdots ,ta^{m-1} \}$.  It is
not hard to see that all of these group elements have $y_0$ for
their $T_2$ coordinate. When $i>0$, the points will share $x_0$ as
their $T_1$ coordinate. Returning to $i \leq 0$, we identify the
$T_1$ coordinates of these points as follows.  Let $x_{i+1}$ be the
point on $l_1$ at height $-(i+1)$. Then the $3^{i+1}$ points which
can be reached in $T_1$ by a path of length $i+1$ which increases in
height at each step are precisely all the points of the form given
in (\ref{eqn:lightson}).  The points at height zero in the tree
$T_1$ which are shown in Figure \ref{fig:DL33}, when paired with
$y_0 \in T_2$, give the complete collection of elements of $L_3$
with bulbs illuminated in all combinations of positions between $-2$
and zero, inclusive.

\subsection{Results of A. Eskin, D. Fisher and K. Whyte}
In our proofs below, we rely on results of A. Eskin, D. Fisher and
K. Whyte concerning the classification of the Diestel-Leader graphs
up to quasi-isometry. \cite{EFW} These results are relevant since
any automorphism of $\Z_n \wr \Z$ can be viewed as a quasi-isometry
of the Cayley graph $DL(n,n)$.

Introduce the following coordinates on $DL(m,n)$: a point in
$DL(m,n)$ can be uniquely identified by the triple $(x,y,z)$ where
$x \in \Q_m, \ y \in \Q_n$ and $ z \in \Z$.  Here, $\Q_m$ denotes
the $m$-adic rationals, which we view as the ends of the tree $T_1$
once the height function has been fixed.  The $n$-adic rationals
correspond in an analogous way to the ends of $T_2$.  Restricting to
$T_1$, any $ x \in \Q_m$ determines a unique line in $T_1$ on which
the height function is strictly increasing.  This line has a unique
point at height $z$.  Similarly, the line in $T_2$ determined by $y
\in \Q_n$ contains a unique point at height $-z$.  This pair of
points together forms a single point in $DL(m,n)$ that is uniquely
identified in this way.

Following \cite{EFW}, we define a {\em product map} $\hat{\vp}:
DL(m,n) \rightarrow DL(m',n')$ as follows. Namely, $\hat{\vp}$ is a
product map if it is within a bounded distance of a map of the form
$(x,y,z) \mapsto (f(x),g(y),q(z))$ or $(x,y,z) \mapsto
(g(y),f(x),q(z))$ where $f:\Q_m \rightarrow \Q_{m'}$ (or $\Q_{n'}$),
$g:\Q_n \rightarrow \Q_{n'}$ (or $\Q_{m'}$) and $q: \R \rightarrow
\R$.

A product map is called {\em standard} if it is the composition an
isometry and a product map in which $q$ is the identity, and $f$ and
$g$ are bilipshitz.

They determine the form of any quasi-isometry between Diestel-Leader
graphs in the following theorem, which is stated in \cite{EFW} and
proven there for $m \neq n$, and proven in the remaining case in
\cite{EFW2}.
\begin{theorem}[\cite{EFW}, Theorem 2.3]
 For any $m\leq n$, any ($K,C$)-quasi-isometry $\vp$ from $DL(m, n)$
to $DL(m', n')$ is within bounded distance of a height respecting
quasi-isometry $\vp'$. Furthermore, the bound is uniform in $K$ and
$C$.
\end{theorem}

Below, we use the fact that any height-respecting quasi-isometry is
at a bounded distance from a standard map. Eskin, Fisher and Whyte
describe the group of standard maps of $DL(m,n)$ to itself; this
group is isomorphic to $(Bilip(\Q_m)\times Bilip(\Q_n)) \rtimes
\Z/2\Z$ when $m = n$ and $(Bilip(\Q_m)\times Bilip(\Q_n))$
otherwise. They deduce that this is the quasi-isometry group of
$DL(m,n)$.  The extra factor of $\Z_2$ in the case $m=n$ reflects
the fact that the trees in this case may be interchanged.

This result has the following implications for our work below.  The
product map structure implies that a quasi-isometry takes a line in
$T_1$ and maps it to within a uniformly bounded distance of a line
in either $T_1$ or $T_2$, depending on whether the quasi-isometry
interchanges the two tree factors or not.  The same is true when we
begin with a line in $T_2$.  This is a geometric fact not dependent
on the group structure, and we use it repeatedly below.

\section{Counting twisted conjugacy classes in $L_n = \Z_n \wr \Z$}
\label{sec:lamplighter}

In this section, we present a geometric proof of the following
theorem, first proven in \cite{GW1}.  We follow the outline of the
lemmas in \cite{GW1}, providing proofs based on the geometry of the
group and the results of A. Eskin, D. Fisher and K. Whyte in
\cite{EFW}.

\begin{theorem}[\cite{GW1}, Theorem 2.3]
\label{thm:Ln} Let $m \geq 2$ be a positive integer with its prime
decomposition
$$m = 2^{e_1}3^{e_2} \Pi_{i} p_i^{e_i}$$
where each $p_i$ is a prime greater than or equal to $5$.  Then
there exists an automorphism $\vp$ of $\Z_m \wr \Z$ with $R(\vp) <
\infty$ if and only if $e_1 = e_2 = 0$.
\end{theorem}

Let $\vp \in Aut(L_n)$. We first must show that the infinite direct
sum $\bigoplus A_i$ where each $A_i \cong \Z_n$ is characteristic in
$L_n$, regardless of whether $(n,6)=1$, allowing us to obtain
``vertical" maps on the short exact sequence defining $L_n$ given in
Equation \ref{eqn:Ln}. Let $DL(n,n)$ be the Cayley graph of $L_n$
with respect to the generating set $\{t,ta,ta^2, \cdots ,ta^{n-1}
\}$. We label the identity using the coordinates $(x_0,y_0)$ and
label other points as needed.  In general, we use lower case letters
for coordinates in the first tree $T_1$, and capital letters to
denote coordinates in the second tree $T_2$.

\begin{lemma}  \label{lemma:characteristic}
Let $\vp \in Aut(L_n)$.  Then $\vp(\bigoplus A_i) \subset \bigoplus
A_i$.
\end{lemma}

\begin{proof}
Let $a = t^{-1}(ta) = (b,y_0)$.  Since $\vp(Id) = \vp(x_0,y_0) =
(x_0,y_0)$, we have one of two situations:  a quasi-isometry of
$DL(n,n)$ which interchanges the tree factors will take any line in
$T_1$ through $x_0$ to within a uniformly bounded distance of a line
in $T_2$ passing through $y_0$, and a quasi-isometry which does not
interchange the tree factors will coarsely preserve the set of lines
in $T_1$ through $x_0$.

Respectively, we have that $\vp(a)=\vp(b,y_0) = (x_0,C)$ or
$\vp(b,y_0) = (c,y_0)$.  In either case it is clear that $\vp(a)$
must lie at height zero in $DL(n,n)$, that is,  $\vp(a) \in
\bigoplus A_i$. If $a_i$ is the generator of the $i$-th copy of $A$
in $\bigoplus A_i$, then $a_i = t^ia_0t^{-i}$, and it is easy to see
that $a_i$ either has coordinates $(x_0,D)$ or $(d,x_0)$ in
$DL(n,n)$.  Thus the same argument shows that $\vp(a_i) \in
\bigoplus A_i$ and the lemma follows.
\end{proof}

Lemma \ref{lemma:characteristic} guarantees the following
commutative diagram:

\begin{equation}\label{short-exact}
\begin{CD}
    1   @>>> \bigoplus A_i    @>>>  L_n @>>>      \Z @>>> 1 \\
    @.  @V{\ph'}VV      @V{\ph}VV   @V{\overline \ph}VV @.\\
    1   @>>> \bigoplus A_i    @>>>  L_n @>>>      \Z @>>> 1
 \end{CD}
\end{equation}
In particular, if $\Op(t) = t$, it follows immediately from Lemma
\ref{lemma:reid} that $R(\vp) = \infty$.  In all that follows, we
assume that $\Op(t) = -t$, which necessarily means that $\vp$
interchanges the tree factors in $DL(n,n)$. Since the fixed point
set of $\Op$ is trivial, if we can show that $R(\vp') = \infty$ it
will follow from Lemma \ref{lemma:reid} that $R(\vp) = \infty$ as
well.

When $n$ is prime, we are able to determine the image of the
generator $a_i$ of the $i$-th copy of $A_i$ in $\bigoplus A_i$ under
any automorphism $\vp \in Aut(L_n)$ using only the geometry of the
group. The simplest proof of this lemma uses the ``lamplighter"
picture describing group elements.  The lemma is proven for all $n$
in \cite{GW1}.

\begin{lemma}\label{lemma:ai}
Let $G = \Z_p \wr \Z = \left(\bigoplus \Z_p \right) \rtimes \Z$
where $p$ is prime, and let $\vp$ be any automorphism of $G$. Denote
by $a_i$ the generator of the $i$-th copy of $\Z_p$ in $\bigoplus
\left(\Z_p\right)_i$.  For each $j \in \Z, \ \exists i \in \Z$ and
$k \in \{1,2, \cdots ,n-1\}$ so that $\vp'(a_i^k) = a_j$.
\end{lemma}

\begin{proof}
Let $\vp'(a_0) = a_{i_1}^{c_1}a_{i_2}^{c_2} \cdots a_{i_k}^{c_k}
a_{-j_1}^{d_1} a_{-j_2}^{d_2} \cdots a_{-j_l}^{d_l}$.  This
expression yields a particular collection of bulbs that are
illuminated in the lamplighter picture corresponding to this
element, with $a_{min}$ and $a_{max}$, respectively, the left- and
right-most of these illuminated bulbs. Then $\vp'(a_i) = \vp'(t^i
\cdot a_0) = \overline{\vp}(t^i) \cdot \vp'(a_0)$.  Since
$\overline{\vp}(t^i)$ subtracts $i$ from the indices of all terms in
$\vp'(a_0)$, we see that $\vp'(a_i)$ is just the same series of
illuminated bulbs as in $\vp'(a_0)$ shifted $i$ units to the left.

Suppose that $\vp'(x) = a_j$, where $x=a_{n_1}^{e_1}a_{n_2}^{e_2}
\cdots a_{n_k}^{e_k} a_{-m_1}^{f_1} a_{-m_2}^{f_2} \cdots
a_{-m_l}^{f_l}$.  Since $p$ is prime, no exponents in
$\vp'(a_{n_i}^{e_i})$ or $\vp'(a_{-m_i}^{f_i})$ become congruent to
$0$ mod $p$, and thus each element of the form $\vp'(a_{n_i}^{e_i})$
or $\vp'(a_{-m_i}^{f_i})$ has the same number of illuminated bulbs
in its lamplighter picture as $\vp(a_0)$.

If we assume that $a_{min} \neq a_{max}$, it is clear that $\vp'(x)$
has illuminated bulbs in positions $min -m_l$ and $max + n_k$ and
thus cannot be equal to $a_j$. We conclude that both $a_{min} =
a_{max}$ and $x = a_i^k$ for some $i \in \Z$.
\end{proof}

Since $p$ is prime, by taking powers of $a_i^k$ it follows from
Lemma \ref{lemma:ai} that $\vp'(a_i) = a_j^r$. Thus, if $\vp'(a_0) =
a_j^k$, then $\vp'(a_i) = a^k_{j-i}$ and $\vp'(a_{j-i}) = a_i^k$ for
the same value of $k$, and we obtain by restriction a map $\vp':A_i
\bigoplus A_{j-i} \rightarrow A_i \bigoplus A_{j-i}$.

We now recall the definition of $\vp'$-twisted conjugacy classes in
$\bigoplus A_i$. If $g_1, \ g_2 \in \bigoplus A_i$ are
$\vp'$-twisted conjugate, then there exists $h \in \bigoplus A_i$ so
that $h g_1 \vp'(h)^{-1} = g_2$.  Since $\bigoplus A_i$ is abelian,
we rewrite this equation as $h \vp'(h)^{-1} = g_1^{-1} g_2$.
Switching to additive notation, we see that if $Id-\vp':A_i
\bigoplus A_{j-i} \rightarrow A_i \bigoplus A_{j-i}$ is invertible,
then $R(\vp') < \infty$.  In particular, if we can find a fixed
point of $\vp'$, we guarantee that $Id-\vp'$ is not surjective and
hence not invertible. In the proof of Theorem \ref{thm:Ln}, we prove
that $\vp'$ must have a fixed point if $n$ is $2$ or $3$, and if
$(n,2) = (n,3) = 1$ then it is possible to construct automorphisms
$\vp'$ so that $Id-\vp'$ is surjective.  This geometric method will
allow us to prove that $L_2$ and $L_3$ have property $\Rinf$, as
well as prove that when $(n,6) = 1$, $L_n$ does not have this
property. When $(n,6) \neq 1$, it will follow from the following
lemma that $L_n$ does have property $\Rinf$.

\begin{lemma}
\label{lemma:characteristic2} If $(n,6)\ne 1$ then there is a
characteristic subgroup $H_i \triangleleft L_n$ with quotient
$\mathbb Z_{\frac{n}{i}}$ where $i|n$, for $i=2$ or $3$, and $H_i
\cong L_i$.
\end{lemma}

\begin{proof}
Suppose that $2|n$.  Let $i=2$, and if $\tau$ is the generator for
$\mathbb Z_n$ then the element $\tau^{\frac{n}{2}}$ is the unique
element of order $2$ in $\mathbb Z_n$. Consider the subgroup
$H_2=\langle \tau^{\frac{n}{2}} \rangle \wr \mathbb Z$ of $L_n$,
where $\Z= \langle t \rangle$. It is clear that $H_2$ is isomorphic
to $L_2$ and that $H_2$ contains $\bigoplus_j (\langle
\tau^{\frac{n}{2}} \rangle)_j$, the subgroup of $L_n$ consisting of
elements of order $2$. It is easy to see that every automorphism
$\vp \in Aut(L_n)$ sends $\bigoplus_j (\langle \tau^{\frac{n}{2}}
\rangle)_j$ to itself, and the subgroup generated by $(1,t)$ to
itself, where $1 \in \bigoplus_j (\langle \tau^{\frac{n}{2}}
\rangle)_j$ and $t \in \Z$. Since $H_2$ is generated by
$(\tau^{\frac{n}{2}},1)$ and $(1,t)$, it follows that $H_2$ is
characteristic in $L_n$. Now $L_n/H_2 \cong \mathbb
Z_{\frac{n}{2}}$. 
The case when $3|n$ and we let $i=3$ is similar.
\end{proof}

One can also show that if $(n,6)\ne 1$ then there is a
characteristic subgroup $K_i$ of $L_n$ such that $L_n/K_i \cong L_i$
for $i=2$ or $3$.

\noindent {\it Proof of Theorem \ref{thm:Ln}.} Consider $\vp \in
Aut(L_n)$, and suppose that $\vp'(a_0) = a_j^k$. Since $\vp$ is an
automorphism, it is also a quasi-isometry.  By \cite{EFW} it is a
uniformly bounded distance from a product map. Let $K$ and $C$ be
the quasi-isometry constants of this product map. We note that since
we are assuming that $\Op(t) = -t$, this product map interchanges
the two tree factors, that is, it maps lines in the first tree
factor in $DL(n,n)$ to within a uniformly bounded distance $D$ of
lines in the second tree factor, where the constant $D$ depends on
$K$ and $C$.

We know that $\vp': \bigoplus A_i \rightarrow \bigoplus A_i$, and
our assumption that $\vp'(a_0) = a_j^k$ yields the restriction
$\vp': A_i \bigoplus A_{j-i} \rightarrow A_i \bigoplus A_{j-i}$. We
call this subset of $\bigoplus A_i$ a {\em block}. We now explain
the effect on these blocks of the geometric fact that $\vp$ is a
uniformly bounded distance from product map. Recall that the value
of $j$ is determined by the image of $a_0$, that is, $\vp'(a_0) =
a_j^k$.  We always consider $i>j$, so that one factor in the block
has a positive index, and the other a negative index. This ensures
that coordinates can be introduced on the $n^2$ points in this block
so that if $a_i^{m_1} = (x_0,B)$ and $a_{j-i}^{m_2} = (c,y_0)$, then
$a_i^{m_1}a_{j-i}^{m_2}$ will have coordinates $(c,B)$.

Since $\vp$ is within a uniformly bounded distance $D$ of a product
map, we additionally choose $i>KD$, where $K$ is the maximum of the
bilipshitz constants of the maps on the tree factors. (Note that
this lower bound is perhaps much larger than required.) This ensures
that when we consider the coordinates in $T_2$ of the points $a_i^k$
and in $T_1$ of the points $a_{j-i}^k$ for $0 \leq k \leq n-1$,
these coordinates are at least distance $2KD$ apart, using the
standard metric on the tree which assigns each edge length one. This
is easy to see by writing $a_i^k = t^{(i-1)} ta^k t^{-i}$; when
$i<0$, the second coordinates of these points, as $k$ varies, is
always $y_0$. The first coordinates, pairwise, are distance $2i>2KD$
apart in $T_1$. More importantly, if $l$ is any line in $T_2$
through the second coordinate of $a_i^k$ on which the height
function is strictly decreasing,  this product map takes $l$ to
within $D$ of a line in $T_1$; the $D$-neighborhood of this line, by
construction, can contain a unique first coordinate of one of the
points in $A_{j-i}$.

Thus the fact that lines in $T_1$ are mapped to within a uniformly
bounded distance of lines in $T_2$ tells us that there is a
surjective function $f_1$ from the set of $T_1$ coordinates of the
points in $A_{j-i}$ to the set of $T_2$ coordinates of the points in
$A_{i}$. In addition, there is a surjective function $f_2$ from the
set of $T_2$ coordinates of the points in $A_{i}$ to the set of
$T_1$ coordinates of the points in $A_{j-i}$.  Moreover, these maps
combine in the following way, as follows from the definition of a
product map given in \cite{EFW}.  If $(x,Y) \in \Ai$, then
$\vp'(x,Y) = (f_2(Y),f_1(x))$.

We begin with the case  $n=2$.  Consider the restriction $\vp':A_i
\bigoplus A_{j-i} \rightarrow A_i \bigoplus A_{j-i}$, where $j$ is
defined by the image of $a_0$ under $\vp'$. Then the set $A_i
\bigoplus A_{j-i}$ consists of four points. Assume that $i >j$ and
that $i$ was chosen to be sufficiently large, as described above.
The first choice ensures that the index of one factor is negative,
and the other is positive; the second choice guarantees that the
points within the block $A_i \bigoplus A_{j-i}$ have coordinates in
the appropriate trees which are at least $D$ units apart from each
other.   We describe the points in $\Ai$ using both the lamplighter
picture of $L_2$ as well as giving coordinates in the graph
$DL(2,2)$, as follows:
\begin{enumerate}
\item the identity  $(x_0,y_0)$
\item a single bulb illuminated in position $i$, where we
introduce the coordinates $a_i = (x_0,B)$
\item a single bulb illuminated in position  $j-i$, where we
introduce the coordinates $a_{j-i} = (c,y_0)$
\item bulbs illuminated in positions $i$ and $j-i$, which
necessarily has coordinates $(c,B)$.
\end{enumerate}
Note that it is exactly the fact that we ensured that $i$ and $j-i$
have opposite signs which determines the coordinates of $a_ia_{j-i}$
to be $(c,B)$.

We use the geometric result of \cite{EFW} that $\vp$ is a uniformly
bounded distance from a product map, which is bilipshitz on each
tree and then interchanges the tree factors, to show that $(c,B)$ is
a fixed point for $\vp'$.  Let $f_1$ and $f_2$ be the functions on
the coordinates of the points of $\Ai$ defined above.  We must have
$f_1(x_0) = y_0$ and $f_2(y_0) = x_0$ since the identity $(x_0,y_0)$
is preserved under any automorphism. Lemma \ref{lemma:ai} implies
that $f_1(B) = c$ and $f_2(c) = B$. This forces $\vp'(c,B) = (c,B)$;
the existence of a fixed point for $\vp'$ says that $Id-\vp'|_{A_i
\bigoplus A_{j-i}}$ is not invertible when $n=2$.  Thus within each
block of the form $\Ai$, for $i>j$, there are at least two distinct
$\vp'$-twisted conjugacy classes.  By choosing group elements in
$\bigoplus A_i$ with entries in the different $\vp'$-twisted
conjugacy classes within each block, we can create infinitely many
different $\vp'$-twisted conjugacy classes in $\bigoplus A_i$. We
conclude that $R(\vp')$ and hence $R(\vp)$ are infinite.

In the case $n=3$, the blocks $\Ai$ have nine elements, and we have
the additional algebraic fact that $a_i^2 = a_i^{-1}$.  We have two
choices for the maps $f_1$ and $f_2$, as follows.
\begin{enumerate}
\item If $\vp'(a_i) = a_{j-i}$ and $\vp'(a_{j-i}) = a_{i}$ then it
is clear that $f_1(c_1) = B_1$ and $f_2(B_1) = c_1$, resulting in
$(c_1,B_1)$ being fixed under $\vp'$.
\item If $\vp'(a_i) = a_{j-i}^2$ and $\vp'(a_{j-i}) = a_{i}^2$ then
we know that $f_1(c_1) = B_2$ and $f_2(B_1) = c_2$.  This forces
$f_1(c_2) = B_1$ and $f_2(B_2) = c_1$.  We see immediately that
$(c_1,B_2)$ is fixed under $\vp'$.
\end{enumerate}
In either case, the existence of a non-trivial fixed point ensures
that $L_3$ has property $\Rinf$.

When $n>3$ is odd and $(n,3) = 1$,  we use the blocks $A_{j-i}
\bigoplus A_{i}$ to construct  $\vp': A_i \bigoplus A_{j-i}
\rightarrow A_i \bigoplus A_{j-i}$ with no nontrivial fixed points,
under the assumption that $\vp'(a_0) = a_j^k$. Thus the map $Id -
\vp'$ will be surjective, hence invertible, on each block, and we
see that there must be a single $\vp'$-twisted conjugacy class in
$\bigoplus A_i$. This then implies that $L_n$ does not have property
$\Rinf$. We begin by assuming that $\vp': \bigoplus A_i \rightarrow
\bigoplus A_i$ is an automorphism with $\vp'(a_0) = a_j^k$, and that
$\overline{\vp}(t): \Z \rightarrow \Z$ is given by
$\overline{\vp}(t) = t^{-1}$. The compatibility condition given in
Equation \ref{eqn:compatibility} ensures that these automorphisms
together induce an automorphism of $L_n = \Z_n \wr \Z$.

Since $\vp'(a_0) = a_m^r$, we again obtain a map on blocks, where
the block $\Ai \bigoplus A_{j-i}$ consists of $n^2$ points; we
introduce coordinates on these points as follows.
\begin{center}
\begin{tabular}{|c|c|c|c|c|c|}
  \hline
   & $a_i^0$ & $a_i$ & $a_i^2$ & $\cdots$ & $a_i^{n-1}$ \\
  \hline
  $a_{j-i}^0$ & $(x_0,y_0)$ & $(x_0,B_1)$ & $(x_0,B_2)$ & $\cdots$ & $(x_0,B_{n-1})$ \\
  $a_{j-i}$ &  $(c_1,y_0)$ & $(c_1,B_1)$ & $(c_1,B_2)$ & $\cdots$ & $(c_1,B_{n-1})$ \\
  $a_{j-i}^2$ & $(c_2,y_0)$ & $(c_2,B_1)$ & $(c_2,B_2)$ & $\cdots$ & $(c_2,B_{n-1})$\\
  $\vdots$ & $\vdots$ & $\vdots$ & $\vdots$ & $\vdots$ &$\vdots$ \\
  $a_{j-i}^{n-1}$ & $(c_{n-1},y_0)$ & $(c_{n-1},B_1)$ & $(c_{n-1},B_2)$ & $\cdots$ & $(c_{n-1},B_{n-1})$\\
  \hline
\end{tabular}
\end{center}
These points determine $n$ points in each tree, all at height zero,
namely $\{x_0,c_1, \cdots ,c_{n-1} \}$ in $T_1$ and $\{y_0,B_1,
\cdots ,B_{n-1}\}$ in $T_2$.  Choose $i> \max \{j,KD\}$ as above.

We now show that when $(n,2) = (n,3)=1$, we can construct maps $f_1$
and $f_2$ so that the induced automorphism $\vp': \Ai \rightarrow
\Ai$ has no non-trivial fixed points, and thus $Id-\vp':A_i
\bigoplus A_{j-i} \rightarrow A_i \bigoplus A_{j-i}$ is invertible,
which implies that $R(\vp') < \infty$. Suppose that $\vp'(a_i) =
a_{j-i}^2$ and $\vp'(a_{j-i}) = a_i^2$. To obtain a fixed point, we
must solve the equations $\vp'(a_i^k) = a_{j-i}^{2k(\text{mod} \
n)}$ and $\vp'(a_{j-i}^{2k(\text{mod} \ n)}) = a_{i}^k$ for $k$,
that is, $3k \equiv 0 (\text{mod} \ n)$ which has a solution unless
$3|n$.  Thus when $(n,3) = 1$, the map $1-\vp': A_i \rightarrow A_i$
is invertible on each block, hence invertible on $\bigoplus A_i$.

Since we are trying to show that $R(\vp) < \infty$, we refer to part
(4) of Lemma \ref{lemma:reid}.  We know that $R(\overline \ph)=\#
Coker(1-\overline \ph)=2$, so $\overline \ph$ has two twisted
conjugacy classes: $[1]$ and $[t]$. Recalling Equation
\ref{eqn:seq-of-sets}, we must count the twisted conjugacy classes
lying over the classes $[1]$ and $[t]$.  Over $[1]$ we have
$R(\vp')$ twisted conjugacy classes, and over $[t]$ we have
$R(t\cdot \ph')$. Since $Fix \overline \ph=\{1\}$, it follows that
$R(\ph)=R(\ph')+R(t\cdot \ph')$. The action of $t$ shifts the block
$A_i \bigoplus A_{j-i}$ by increasing all the indices by one, and
thus the geometry of the lamplighter picture for $L_n$ easily
implies that $R(\ph')=R(t\cdot \ph')$ and hence
$R(\ph)=2R(\ph')<\infty$.

When $2|n$, we construct the following short exact sequence of
groups:
$$1 \rightarrow \Z_2\wr \Z  \rightarrow \Z_n \wr \Z \rightarrow  \Z_{\frac{n}{2}} \rightarrow 0$$
It is proven in Lemma \ref{lemma:characteristic} that $\Z_2\wr \Z$
is characteristic in $\Z_n \wr \Z$, allowing us to obtain a
commutative diagram based on the above short exact sequence and any
$\vp \in Aut(\Z_n \wr \Z)$. We proved above that $\Z_2 \wr \Z$ has
property $\Rinf$ and it follows from Lemma \ref{lemma:reid} $(2)$
that $\Z_n \wr \Z$ has property $\Rinf$ as well. The argument is
analogous when $3|n$. \qed

When $G$ is a finite abelian group, $G \wr \Z$ has as its Cayley
graph $DL(|G|,|G|)$, with respect to the generating set
$\{t,tg_1,tg_2, \cdots ,tg_{n-1}\}$ where $\{g_i\}$ represent all
elements of $G$.  Gon\c calves and Wong \cite{GW1} prove the following
general result concerning which groups of this form have property
$\Rinf$.

\begin{theorem}[\cite{GW1}, Theorem 3.7]
\label{thm:abeliancase} Let $$G = \bigoplus_j
(\Z_{p_j^{k_j}})^{r_j}$$ be a finite abelian group, where the $p_j$
are distinct primes.  Then for all automorphisms $\vp \in Aut(G)$,
$R(\vp) = \infty$ if and only if $p_j = 2$ or $3$ for some $j$ and
$r_j = 1$.
\end{theorem}


\section{Counting twisted conjugacy classes in $G \wr \Z$}

We end this paper by studying property $R_{\infty}$ for general
lamplighter groups of the form $G\wr \mathbb Z$, where $G$ is an
arbitrary finite group.  The Cayley graph of $G \wr \Z$ is again the
Diestel-Leader graph $DL(n,n)$, where $n = |G|$, with respect to the
generating set $\{tg|g \in G\}$.

We prove several results about when these general lamplighter groups
have property $\Rinf$, but do not give a complete classification of
which groups of this form have or do not have the property.  The
first two results are algebraic, using various expressions of $G$ as
a group extension, and we end with geometric results relying on
$DL(n,n)$.  We first note that as for $\Z_n \wr \Z$, we have the
following short exact sequence in which $\bigoplus G_i$ is
characteristic:
$$1 \rightarrow \bigoplus G_i \rightarrow G \wr \Z \rightarrow \Z
\rightarrow 1$$

\begin{lemma}\label{lemma:+G-characteristic}
Let $\vp \in Aut(G \wr \Z)$.  Then $\vp(\bigoplus G_i) \subset
\bigoplus G_i.$
\end{lemma}
\begin{proof}
Let $|G|=n$ and put coordinates on $DL(n,n)$ as we did in the case
of the lamplighter group $L_n = \Z_n \wr \Z$: denote the identity by
$(x_0,y_0)$ and the generators of $G_0$, the copy of $G$ in
$\bigoplus G$ indexed by $0$, by $(x_0,B_i)$.

Let $\vp \in Aut(G \wr \Z)$.  The argument given in Lemma
\ref{lemma:ai} quoting the result on \cite{EFW} again applies, and
we conclude that $\vp(x_0,B_i) = (c_i,y_0)$, for some coordinates
$c_i$. Regardless of the actual vertices of the trees represented by
the $c_i$, we conclude that $\vp(x_0,B_i)$ lies at height zero in
$DL(n,n)$, that is, $\vp(x_0,B_i) \in \bigoplus G_i$. Since
$\vp(G_0) \subset \bigoplus G_i$ and the action of $t$ is to
translate between the factors of $G$ in the sum $\bigoplus G_i$, we
see that $\bigoplus G_i$ is characteristic in $G \wr \Z$.
\end{proof}

Applying Lemma \ref{lemma:+G-characteristic}, we see that the
following diagram is commutative, for $\phi \in Aut(G \wr \Z)$.
\begin{equation}\label{general-type}
\begin{CD}
    1   @>>> \bigoplus G_i    @>>>  G\wr \Z @>>>       \Z @>>> 1 \\
    @.  @V{\ph'}VV      @V{\ph}VV   @V{\overline \ph}VV @.\\
    1   @>>> \bigoplus G_i    @>>>  G\wr \Z @>>>       \Z @>>> 1
 \end{CD}
\end{equation}

Let $G = \{Id,g_1,g_2, \cdots g_n\}$.  To denote the elements of the
$i$-th copy of $G$ in $\bigoplus G_i$, we use the notation $g_{i,j}$
where $i$ is the index of $G_i$ in $\bigoplus G_i$ and $j \in \{1,2,
\cdots ,n\}$.  The symbol $Id_i$ will denote the identity in $G_i$.

If $G$ itself can be written as a group extension $1 \rightarrow A
\rightarrow  G \rightarrow  C \rightarrow  1$ then one always
obtains the following short exact sequences, although in both cases
the kernel is not necessarily a characteristic subgroup of $G \wr
\Z$:

\begin{equation}\label{type1}
1 \rightarrow \bigoplus A_i \rightarrow  G \wr \Z \rightarrow  C \wr
\Z \rightarrow 1 \hspace{.5in} \text{and} \hspace{.5in} 1
\rightarrow A \wr \Z \rightarrow G \wr \Z \rightarrow  C \rightarrow
1
\end{equation}

Given an arbitrary finite group $G$, we have two natural short exact
sequences:
$$1 \rightarrow [G,G] \rightarrow  G  \rightarrow  G^{Ab} \rightarrow 1 \hspace{.5in} \text{and} \hspace{.5in} 1
\rightarrow Z(G) \rightarrow G  \rightarrow  G/Z(G) \rightarrow 1$$
where $[G,G]$ and $Z(G)$ denote the commutator subgroup and the
center of $G$, respectively.  Both of these subgroups are
characteristic in $G$.  To obtain the related commutative diagrams,
we must first prove the following lemma.

\begin{lemma}\label{lemma:commutator-center-char}
The subgroups $\bigoplus ([G,G])_i$ and $Z(G) \wr \Z$ of $G \wr \Z$
are both characteristic.
\end{lemma}

\begin{proof}
We show that $\bigoplus ([G,G])_i$  is characteristic in $\bigoplus
Gi$ and thus in $G \wr \Z$.  The commutator subgroup $[\bigoplus
G_i, \bigoplus G_i]$ has the following form.  Let $a,b \in \bigoplus
G_i.$  Then $$a = \sum_s g_{\sigma(s),j_s}  \text{  and  } b =
\sum_t g_{\tau(t),j_t}$$ where $\sigma$ and $\tau$ are injective
functions from $\{1,2,3, \cdots ,k\}$ and $\{1,2, \cdots ,l\}$,
respectively, to $\Z$ so that $\sigma(1)<\sigma(2) < \cdots <
\sigma(k)$ and $\tau(1) < \tau(2) < \cdots < \tau(l)$.  When we form
$[a,b]$, one of two things must occur:
\begin{enumerate}
\item If $\sigma(s) = \tau(t)$ for some $t$, then the coordinate of
$[a,b]$ in $G_{\sigma(s)}$ is $[g_{\sigma(s),j_s},g_{\tau(t),j_t}]$.
\item If $\sigma(s) \neq \tau(t)$ for any $t$, then the coordinate
of $[a,b]$ in $G_{\sigma(s)}$ is
$g_{\sigma(s),j_s}g_{\sigma(s),j_s}^{-1} = Id_{\sigma(s)}$.  The
same is true for indices $\tau(t)$ which are not equal to
$\sigma(s)$ for any $s$.
\end{enumerate}
Since each non-identity coordinate of $[a,b]$ is a commutator of
$G$, we see that $[\bigoplus G_i, \bigoplus G_i] \subset \bigoplus
([G,G])_i$, and the other inclusion is clear.  Thus $\bigoplus
([G,G])_i$ must be a characteristic subgroup.

To show that $Z(G) \wr \Z$ is characteristic in $G \wr \Z$, we first
show that $\bigoplus Z(G)_i$ is characteristic in $\bigoplus G_i$.
Since the group operation between elements of $\bigoplus G_i$
involves component-wise multiplication, it follows immediately that
$Z(\bigoplus G_i) = \bigoplus Z(G)_i$, and thus this subgroup is
characteristic.

Let $\vp \in Aut(G \wr \Z)$.  We obtain automorphisms $\vp':
\bigoplus G_i \rightarrow \bigoplus G_i$ and $\overline{\vp}: \Z
\rightarrow \Z$; we assume the latter is given by $\overline{\vp}(t)
= t^{-1}$. Since $\bigoplus Z(G)_i$ is characteristic in $\bigoplus
G_i$, we can restrict to obtain $\vp': \bigoplus Z(G)_i \rightarrow
\bigoplus Z(G)_i$.  Since $\vp'$ and $\overline{\vp}$ satisfy the
compatibility condition given in equation \ref{eqn:compatibility},
we induce a map on $Z(G) \wr \Z$ which makes the following diagram
commute and must be the restriction of $\vp$ to $Z(G) \wr \Z$:
\begin{equation}
\begin{CD}
    1   @>>> \bigoplus Z(G)_i    @>>>  Z(G)\wr \Z @>>>      \Z @>>> 1 \\
    @.  @V{\ph'}VV      @V{\ph}VV   @V{\overline \ph}VV @.\\
    1   @>>> \bigoplus Z(G)_i   @>>>  Z(G)\wr \Z @>>>      \Z @>>> 1
 \end{CD}
\end{equation}
From this we conclude that $Z(G) \wr \Z$ is a characteristic
subgroup of $G \wr \Z$.
\end{proof}

It follows from Lemma \ref{lemma:commutator-center-char} that given
any $\varphi \in Aut(G\wr \mathbb Z)$, we obtain the following
commutative diagrams:

\begin{equation}\label{typeI}
\begin{CD}
    1   @>>> \bigoplus ([G,G])_i    @>>>  G\wr \Z @>>>      G^{Ab}\wr \Z @>>> 1 \\
    @.  @V{\ph'}VV      @V{\ph}VV   @V{\overline \ph}VV @.\\
    1   @>>> \bigoplus ([G,G])_i    @>>>  G\wr \Z @>>>      G^{Ab}\wr \Z @>>> 1
 \end{CD}
\end{equation}

which is a special case of (\ref{type1}) given above, and

\begin{equation}\label{typeII}
\begin{CD}
    1   @>>> Z(G)\wr \Z    @>>>  G\wr \Z @>>>      G/Z(G) @>>> 1 \\
    @.  @V{\ph'}VV      @V{\ph}VV   @V{\overline \ph}VV @.\\
    1   @>>> Z(G)\wr \Z    @>>>  G\wr \Z @>>>      G/Z(G) @>>> 1
 \end{CD}
\end{equation}
Here, the projection $G\wr \Z \to G/Z(G)$ is given by
$$(\sum_k^m a_{i_k},t^j) \mapsto \prod_k^m [a_{i_k}]$$
where $[a]$ is the image of $a\in G$ in $G/Z(G)$ and $i_1<...<i_m$.

We will always use $\ph'$ and $\overline \ph$ for the induced
homomorphisms on the kernel and on the quotient, respectively.  We
note that the two maps in the diagrams above labeled $\vp'$  (resp.
$\Op$) denote different maps in the two diagrams.

Denote by $\mathfrak A$ the family of finite abelian groups $A$ such
that $A\wr \Z$ has property $R_{\infty}$.  This family was
completely determined in \cite{GW1}, and is listed above as Theorem
\ref{thm:abeliancase}. Similarly, let $\mathfrak L$ be the family of
finite groups $G$ such that $G\wr \Z$ has property $R_{\infty}$.  We
use both the algebra and the geometry of these lamplighter groups to
determine some conditions under which $G \in \mathfrak{L}$.

\begin{theorem}\label{general-lamplighter}
Let $G$ be a finite group. If $G^{Ab}\in \mathfrak A$ or $Z(G)\in
\mathfrak A$ then $G\in \mathfrak L$.
\end{theorem}
\begin{proof} First suppose that $G^{Ab}\in \mathfrak A$. It follows from
\eqref{typeI} and Lemma \ref{lemma:reid} that $G\wr \Z$ has property
$R_{\infty}$, provided that $\overline \ph$ is an automorphism of
$G^{Ab}\wr \Z$. By Lemma \ref{lemma:auto}, it suffices to show that
$G^{Ab}\wr \Z$ is Hopfian. Since $G^{Ab}$ is finite and abelian,
Theorem 3.2  of \cite{G} shows that $G^{Ab}\wr \Z$ is residually
finite. Now, $G^{Ab}\wr \Z$ is finitely generated so we conclude
that it is Hopfian.

Next, suppose that $Z(G)\in \mathfrak A$, which yields $R(\vp') =
\infty$. The inclusion $i:Z(G)\wr \Z \to G\wr \Z$ induces a function
$\hat i:\mathcal R(\varphi') \to \mathcal R(\varphi)$ where
$\mathcal R(\psi)$ denotes the set of $\psi$-twisted conjugacy
classes. Since $G/Z(G)$ is finite, so is the number of fixed points
of $\Op$.  The reasoning in part (4) of Lemma \ref{lemma:reid} shows
that $\mathcal R(\ph)$ is in 1-1 correspondence with $\mathcal
R(\ph')$ modulo the action of $Fix \overline \ph$.  Thus $R(\vp) =
\infty$ as well.
\end{proof}

Using Theorem \ref{thm:abeliancase} one can easily determine when
$Z(G) \in \mathfrak A$.  We now list some examples of groups with
property $\Rinf$ which follow immediately from Theorem
\ref{general-lamplighter}.

\begin{example}[Dihedral groups $D_{2n}$ for odd $n$]\label{D6}
{\em Let $G=D_6\cong \Z_3 \rtimes \Z_2$ be the dihedral group of
order $6$.  In this case, $[G,G]\cong \Z_3$ while $Z(G)=\{1\}$.
Since $G^{ab} \cong \Z_2 \in \mathfrak A$, it follows from Theorem
\ref{general-lamplighter} that $D_6\in \mathfrak L$, that is, $D_6
\wr \Z$ has property $\Rinf$. Moreover, if $n$ is odd and $D_{2n}$
is the dihedral group of order $2n$ then $\Z_n$ is a characteristic
subgroup of $D_{2n}\cong \Z_n\rtimes \Z_2$. Using the commutative
diagram \eqref{general-type} and the fact that $\Z_2 \in \mathfrak
A$, it follows from part (1) of Lemma \ref{lemma:reid} and Theorem
\ref{general-lamplighter} that $D_{2n} \in \mathfrak L$ when $n$ is
odd.}
\end{example}

\begin{example}[Dihedral groups $D_{2n}$ for even $n$]\label{D-even}
{\em When $n$ is even, $D_{2n}\cong \Z_n\rtimes \Z_2$, then $D_{2n}
\wr \Z$ also has property $\Rinf$.  If we take $Z_n \cong \langle t
\rangle$, then the center of $D_{2n}$ is isomorphic to $\Z_2 \cong
\langle t^{\frac{n}{2}}\rangle$.  Thus $Z(D_{2n}) \in \mathfrak A$
and it follows from Theorem \ref{general-lamplighter} that $D_{2n}
\wr \Z$ has property $\Rinf$ for all $n >0$.}
\end{example}

\begin{example}[Quaternion group of order $8$]\label{Q8}
{\em Let $G=Q_8$ be the quaternion group of order $8$. In this case,
$[G,G]=Z(G)\cong \Z_2$. Since $\Z_2\in \mathfrak A$, it follows from
Theorem \ref{general-lamplighter} that $Q_8\in \mathfrak L$.}
\end{example}

Combining Theorem \ref{thm:Ln} with Examples \ref{D6} and
\ref{D-even}, we have proven the following.

\begin{proposition}\label{prop:2p}
If $G$ is any finite group with order $2p$, where $p$ is an odd
prime, then $G \wr \Z$ has property $\Rinf$.
\end{proposition}

\begin{proof}
It is an elementary theorem in algebra that any group of order $2p$,
where $p$ is an odd prime, must be either cyclic or dihedral.
\end{proof}

If $Z(G)$ or $G^{ab}$ is unknown, the following special condition
may be applicable, which allows us to extend our results to some
nilpotent groups, dependent on order.

\begin{theorem}\label{thm:sylow}
Let $G$ be a finite group with a unique Sylow $2$-group $S_2$.  If
$Z(S_2) \in \mathfrak{A}$ then $G \in \mathfrak L$.  Similarly, if
$G$ has a unique Sylow $3$-group $S_3$ with $Z(S_3) \in
\mathfrak{A}$, then $G \in \mathfrak L$.
\end{theorem}

Before proving Theorem \ref{thm:sylow} we state two corollaries.
\begin{corollary}
Let $G$ be a finite nilpotent group whose order is divisible by
either 2 or 3, and let $S_2$ and/or $S_3$ denote, respectively, the
unique Sylow $2$- and/or $3$-subgroups. If $Z(S_i) \in
\mathfrak{A}$, for $i = 2$ or $i=3$, then $G \wr \Z$ has property
$\Rinf$.
\end{corollary}

Some elementary group theory leads to an additional corollary.
\begin{corollary}
Let $p$ and $q$ be prime, and let $G$ be a finite nonabelian group
whose order is one of:
\begin{enumerate}
\item $2p^n$, where $2<p$,
\item $3q^m$, where $3<q$,
\item $2p^nq^m$, where $2 <p<q$, or
\item $3p^nq^m$, where $3 <p<q$.
\end{enumerate}
Then $G \wr \Z$ has property $\Rinf$.
\end{corollary}

\begin{proof}
It is an exercise in group theory to check that groups with the
above orders have either a unique Sylow $2$ subgroup of order $2$ or
a unique Sylow $3$ subgroup of order $3$.  The conclusion then
follows directly from Theorem \ref{thm:sylow}.
\end{proof}

We now return to the proof of Theorem \ref{thm:sylow}.
\begin{proof}[Proof of Theorem \ref{thm:sylow}.]
We prove the theorem in the case that $G$ has a unique Sylow
$2$-subgroup $S_2$; the case for the unique Sylow $3$-subgroup is
analogous.  It follows immediately from Theorem
\ref{general-lamplighter} that $S_2 \wr \Z$ has property $\Rinf$. We
will show that $S_2 \wr \Z$ is a characteristic subgroup of $G \wr
\Z$ and the theorem then follows from Lemma \ref{lemma:reid}, the
short exact sequence $1 \rightarrow S_2 \wr \Z \rightarrow G \wr \Z
\rightarrow G/S_2 \rightarrow 1$, and the fact that a unique Sylow
subgroup is normal.

Let $\vp \in Aut(G \wr \Z)$, $\vp'$ its restriction to $\bigoplus
G_i$ and $\overline{\vp}$ its projection to $\Z$.  Since $S_2$
contains all elements of $G$ whose order is a power of $2$, the same
is true for $\bigoplus (S_2)_i \subset \bigoplus G_i$.  Since order
is preserved under automorphism, $\vp'(\bigoplus (S_2)_i) \subset
\bigoplus (S_2)_i$ and this subgroup is characteristic. We now have
vertical maps from the short exact sequence $1 \rightarrow \bigoplus
(S_2)_i \rightarrow S_2 \wr \Z \rightarrow \Z \rightarrow 1$ where
the map on the kernel is the restriction of $\vp'$ and the map on
the quotient is $\overline{\vp}$.  These maps induce a map on $S_2
\wr \Z$ which must be the restriction of $\vp$ to $S_2 \wr \Z$. Thus
$S_2 \wr \Z$ is a characteristic subgroup of $G \wr \Z$.
\end{proof}

We extend these results further for simple groups, and as a
consequence of Theorem \ref{thm:outer} below, determine that $A_n
\wr \Z$ and $S_n \wr \Z$ have property $\Rinf$ for all $n \geq 5$.

\begin{lemma}\label{lemma:simple}
Let $G$ be a finite simple group, and $\vp \in Aut(G \wr \Z)$.  For
each $j \in \Z$, there exists $i \in \Z$ so that $\vp'(G_i) = G_j$.
\end{lemma}

\begin{proof}
Since $G$ is finite, there is an $r \in \Z^+$ so that $\vp'(G_0)
\subset \cup_{k=1}^r G_{i_k}$, where the $i_k$ are distinct indices.
Consider the homomorphism defined by composing $\vp'|_{G_0}$ with
projection onto the first factor of $\cup_{k=1}^r G_{i_k}$. Since
$G$ is simple, there can be no kernel, so this is an automorphism of
$G$. This argument holds for any factor in $\cup_{k=1}^r G_{i_k}$.
We obtain a set $\{\xi_k\}$ of automorphisms of $G$ so that $\xi_k:G
\rightarrow G_{i_k}$.  We observe that since each $\xi_k$ is an
automorphism, for all nontrivial $g \in G_0$ and for all values of
$k$, the element $\xi_k(g)$ is nontrivial.

As the action  of $t$ on $\bigoplus G_i$ is by translation, the same
set of automorphisms can be used to determine the image of any
element of $G_i$ in $\bigoplus G_i$.

Since we began with a group automorphism, and $\bigoplus G_i$ is a
characteristic subgroup of $G \wr \Z$, there is some $x \in
\bigoplus G_i$ so that $\vp'(x) = g_j$. Now the proof of Lemma
\ref{lemma:ai} allows us to conclude that $x \in G_i$ for some $i$.
Since the action of $t$ on $\bigoplus G_i$ is by translation, there
is a single automorphism $\xi: G \rightarrow G$ which is used to
determine the image of any element in $\bigoplus G_i$ under $\vp \in
Aut(G \wr \Z)$.
\end{proof}

Using Lemma \ref{lemma:simple}, we can now prove the following
theorem.

\begin{theorem}\label{thm:outer}
Let $G$ be a finite simple group whose outer automorphism group is
trivial.  Then $G \in \mathfrak{L}$, that is, $G \wr \Z$ has
property $\Rinf$.
\end{theorem}

\begin{proof}
It follows from Lemma \ref{lemma:simple} that $\vp': \bigoplus G_i
\rightarrow  \bigoplus G_i$ preserves blocks of the form $G_i
\bigoplus G_{j-i}$.  We mimic the proof of Theorem \ref{thm:Ln}, and
consider the restriction of $\vp'$ to each block.  If $G$ has no
outer automorphisms, then $\vp'|_{G_i}$ must be conjugation by the
same group element on each block.  Since conjugation by $g \in G$
always has $g$ as a nontrivial fixed point, the element $(g,g) \in
G_i \bigoplus G_{j-i}$ will be a nontrivial fixed point for $\vp':
G_i \bigoplus G_{j-i} \rightarrow G_i \bigoplus G_{j-i}$. Note that
the number of $\vp$-twisted conjugacy classes need not be the index
of the subgroup $(Id -\vp)$ when the group is non-abelian. Instead,
the number of such classes is given by the number of ordinary
conjugacy classes $[x]$ for which $[x]=[\vp(x)]$ (see e.g.
\cite[Theorem 5]{FH}). In particular, a nontrivial fixed point of
$\vp$ yields a class other than that of the identity element. Thus
there are at least two $\vp$-twisted conjugacy classes on each
block, and the theorem follows.
\end{proof}

Since both $\mathbb Z_2$ and $\mathbb Z_3$ are finite simple groups
with trivial outer automorphism groups, Theorem \ref{thm:outer}
yields the immediate corollary that $L_2$ and $L_3$ have property
$R_{\infty}$ \cite{GW1}.

It follows immediately that many of the finite simple groups lie in
$\mathfrak{L}$, including:
\begin{enumerate}
\item the Matthieu groups $M_{11}, \ M_{23}$ and $M_{24}$,
\item the Conway groups $C_1, \ C_2$ and $C_3$,
\item the Janko groups $J_1$ and $J_4$,
\item the baby monster $B$ and the Fischer-Griess monster $M$,
\item other sporadic groups: the Fischer group $Fi_{23}$, the Held group
$He$, the Harada-Norton group $HN$, the Lyons group $Ly$, and the
Thompson group $Th$.
\end{enumerate}

In the case of $A_n$ for $n=5$ and $n \geq 7$, which has outer
automorphism group isomorphic to $\Z_2$, we obtain the following
corollary.

\begin{corollary}\label{cor:alternating}
The alternating group $A_n$ for $n \geq 5$, $n \neq 6$, is in
$\mathfrak L$, that is, $A_n \wr \Z$ has property $\Rinf$.
\end{corollary}

\begin{proof}
For $n \neq 6$, we have $Out(A_n) = \Z_2$.  This single outer
automorphism is conjugation by an odd permutation, and thus
preserves a conjugacy class within $A_n$.  Since this outer
automorphism is defined up to inner automorphisms, we can choose it
so that there is a fixed point in $A_n$. Then the argument in the
proof of Theorem \ref{thm:outer} works verbatim, that is, the
automorphism $\vp'$ when restricted to the blocks $G_j \bigoplus
G_{i-j}$ again must have a non-trivial fixed point.
\end{proof}

Since $A_n$ is a subgroup of index $2$ in $S_n$, we can apply
reasoning similar to the proof of Lemma \ref{lemma:simple} to obtain
the following proposition.
\begin{proposition}
\label{prop:symmetric} For $n \geq 5$, the symmetric group $S_n \in
\mathfrak{L}$, that is, $S_n \wr \Z$ has property $\Rinf$.
\end{proposition}

\begin{proof}
Since $A_n$ is a subgroup of index $2$ in $S_n$, we obtain the
following short exact sequence:
$$1 \rightarrow \bigoplus (A_n)_i \rightarrow S_n \wr \Z \rightarrow
\Z_2 \wr \Z \rightarrow 1$$
Since $\mathbb Z_2\wr \mathbb Z$ is finitely generated and residually finite by \cite[Theorem 3.2]{G}, it is Hopfian so that Lemma \ref{lemma:auto} applies. Therefore, if we can show that $\bigoplus (A_n)_i$
is a characteristic subgroup of $\bigoplus (S_n)_i$, then the
proposition follows from Lemma \ref{lemma:reid} and Theorem
\ref{thm:Ln}.

Following the proof of Lemma \ref{lemma:simple}, suppose that
$\vp'(G_0) \subset \cup_{k=1}^r G_{i_k}$.  When we compose with the
projection onto any of these $G_{i_k}$ factors, we obtain one of two
possible images for $G \cong S_n$: either $S_n$ or $\Z_2$, where the
latter arises when the kernel of the homomorphism is $A_n$.  In
either case, the image of $A_n$ under the composition of $\vp'$ and
this projection map must be $A_n$.  Now restrict the map $\vp'$ to
$\bigoplus (A_n)_i$ initially and it follows that $\bigoplus
(A_n)_i$ is a characteristic subgroup of $\bigoplus (S_n)_i$.
\end{proof}

The following corollary follows immediately from Proposition
\ref{prop:symmetric}, since any finite group can be embedded into
$S_n$ for some value of $n$.

\begin{corollary}
Every group of the form $G \wr \Z$ where $G$ is a nontrivial finite
group can be embedded into a group which has property $\Rinf$.
\end{corollary}


\begin{thebibliography}{TWW}

\bibitem[B]{B}
R. Brown et al (eds.), Handbook of topological fixed point theory, Springer, 2005.

\bibitem[CT]{CT}
S. Cleary and J. Taback, Dead end words in lamplighter groups and
other wreath products, {\em Quarterly Journal of Mathematics} Volume
56 No. 2 (2005), 165-178.


\bibitem[EFW]{EFW}
A. Eskin, D. Fisher, and K. Whyte, Coarse differentiation of
quasi-isometries I: spaces not quasi-isometric to Cayley graphs,
preprint, 2007.

\bibitem[EFW2]{EFW2}
A. Eskin, D. Fisher, and K. Whyte, Coarse differentiation of
quasi-isometries II: Rigidity for Sol and Lamplighter groups,
preprint, 2007.

\bibitem[F]{F}
A.L. Fel'shtyn, The Reidemeister number of any automorphism of a
Gromov hyperbolic group is infinite, {\em Zap. Nauchn. Sem. POMI}
Vol. 279(2001), 229-241.

\bibitem[FG1]{FG}
A.L. Fel'shtyn and D. Gon\c calves, Twisted conjugacy classes in
Symplectic groups, Mapping class groups and Braid groups (including
an Appendix: Geometric group theory and $R_{\infty}$ property for mapping class groups, written with Francois Dahmani), preprint 2007.

\bibitem[FG2]{FGo}
A.L. Fel'shtyn and D. Gon\c calves, Twisted conjugacy classes of
automorphisms of Baumslag-Solitar groups, preprint 2004.

\bibitem[FH]{FH}
A.L. Fel'shtyn and R. Hill, The Reidemeister zeta function with
applications to Nielsen theory and a connection with Reidemeister
torsion, {\em K-Theory} 8 No. 4 (1994), 367-393.

\bibitem[GW1]{GW1}
D. Gon\c calves and P. Wong, Twisted conjugacy classes in wreath
products,  {\em Internat. J. Alg. Comput.} 16 (2006), 875-886.

\bibitem[GW2]{GW2}
D. Gon\c calves and P. Wong, Twisted conjugacy in exponential growth
groups,  {\em Bull. London Math. Soc.}, Vol. 35 (2003), 261-268.

\bibitem[GW3]{GW3}
D. Gon\c calves and P. Wong, Twisted conjugacy classes in nilpotent
groups, {\em J. Reine Angew. Math.}, to appear.

\bibitem[GZ]{GZ}
R. Grigorchuk and A. Zuk, The lamplighter group as a group generated
by a 2-state automaton, and its spectrum, {\em Geom. Dedicata}, 87
No. 1-3 (2001) 209–244.

\bibitem[G]{G}
K. Gruenberg, Residual properties of infinite soluble groups, {\em Proceedings London Math. Soc. (3)} 7 (1957), 29-62.

\bibitem[J]{J}
B. Jiang, Lectures on Nielsen Fixed Point Theory, Contemp. Math.
{\bf 14}, Amer. Math. Soc., Providence 1983.

\bibitem[L]{L} G. Levitt, On the automorphism group of generalized Baumslag-Solitar groups,
{\em Geometry and Topology} 11 (2007), 473--515.

\bibitem[LL]{LL}
G. Levitt and M. Lustig, Most automorphisms of a hyperbolic group
have simple dynamics, {\em Ann. Sci. Ecole Norm. Sup.} 33 (2000),
507-517.

\bibitem[MN]{MN}
Letter from R. Moeller to W.Woess, 2001.

\bibitem[TWh]{TWh}
J. Taback and K. Whyte, Twisted conjugacy and group actions,
preprint, 2005.

\bibitem[TWo]{TW}
J. Taback and P. Wong, Twisted conjugacy and quasi-isometry
invariance for  generalized solvable Baumslag-Solitar groups,
{\em Journal London Math. Soc. (2)} 75 (2007), 705-717.

\bibitem[TWo2]{TW2}
J. Taback and P. Wong, A note on twisted conjugacy and generalized Baumslag-Solitar groups, arXiv:math.GR/0606284, preprint 2006.

\bibitem[We]{Wecken2}
F. Wecken, Fixpunktklassen. III. Mindestzahlen von Fixpunkten,
(German) {\em Math. Ann.}  118 (1942), 544--577.


\bibitem[Woe]{Woe}
W. Woess, Lamplighters, Diestel-Leader graphs, random walks, and
harmonic functions, {\em Combinatorics, Probability $\&$ Computing}
14 (2005) 415- 433.

\bibitem[Wo]{Wo}
P. Wong, Reidemeister number, Hirsch rank, coincidences on
polycyclic groups and solvmanifolds, {\em J. Reine Angew. Math.},
524(2000), 185-204.

\bibitem[Wo2]{Wo2}
P. Wong, Fixed point theory for homogeneous spaces---a brief survey,
Handbook of topological fixed point theory,  265--283, Springer,
Dordrecht, 2005.

\bibitem[Wo3]{Wo3}
P. Wong, Fixed point theory for homogeneous spaces, II, {\em Fund.
Math.}  186  (2005),  no. 2, 161--175.

\bibitem[Wor]{Wor}
K. Wortman, A finitely presented solvable group with small
quasi-isometry group, Michigan Math. J. 55 (2007), no. 1, 3-24.

\end{thebibliography}
\end{document}